\documentclass[12pt]{amsart} 

\pdfoutput=1

\usepackage{graphicx}
\usepackage{pinlabel}
\usepackage{amsfonts}
\usepackage{amsthm}

\newtheorem{Def}{Definition}
\newtheorem{Prop}[Def]{Proposition}
\newtheorem{Lemma}[Def]{Lemma}

\newtheorem{innercustomthm}{Theorem}
\newenvironment{customthm}[1]
  {\renewcommand\theinnercustomthm{#1}\innercustomthm}
  {\endinnercustomthm}

\title{Splittings of free groups from arcs and curves}
\author{Maxwell Forlini}
\email{msforlini@gmail.com}

\date{\today}

\begin{document}

\maketitle

\begin{abstract}
We show that the arc graph of $S_g^1$ is a coarse Lipschitz retract of the  free splitting complex of $F_{2g}$. We also show that the arc and curve graph of $S_g^1$ is a coarse Lipschitz retract of both the cyclic splitting graph of $F_{2g}$ and the maximally cyclic splitting graph of $F_{2g}$.
\end{abstract}

\section{Introduction}

Let $S_g^1$ be a closed, orientable surface of genus $g\geq 2$ with one boundary component, and fix an identification of $\pi_1(S_g^1)$ with $F_{2g}$, the free group on $2g$ generators. Let $Mod(S_g^1)$ denote the mapping class group of $S_g^1$. Let $f$ be a coarsely well-defined map between a metric space $X$ to subspace $Y$; $f$ is a $(K, C)$-\textit{coarsely Lipschitz retraction} if $f|_Y=id$, and $diam_Y(f(a) \cup f(b)) \leq K d_X(a,b) + C$ for all $a,b\in X$ where $K \geq 1$, $C\geq 0$ are constants. In this paper, we will define a $Mod(S_g^1)$-equivariant map of $\mathcal{A}(S_g^1)$ into $FS_{2g}$ and show that the image of $\mathcal{A}(S_g^1)$ is a $(1,C)$-coarse Lipschitz retract of $FS_{2g}$. In particular, this implies that the map is a quasi-isometric embedding of $\mathcal{A}(S_g^1)$ into $FS_{2g}$. We will also prove an analogous result with $\mathcal{AC}(S_g^1)$ and two cyclic splitting graphs, $FZ_{2g}$ and $FZ^{max}_{2g}$.

\subsection{Splitting Complexes}
 A \textit{splitting}, $T$, of $F_n$ is a simplicial tree along with a minimal simplicial action of $F_n$. Two splittings are equivalent if there exists an $F_n$-equivariant homeomorphism between them. A splitting is a \textit{free} (resp. \textit{cyclic}) \textit{$k$-edge splitting} when there are $k$ orbits of edges and the edge stabilizers are trivial (resp. cyclic). A splitting, $T$, is a refinement of $T'$ if we can obtain $T$ by equivariantly collapsing edges of $T'$. In practice we will work with a splitting by way of the total space, $X_T$. The \textit{total space} of a splitting is a $K(F_n, 1)$ constructed as follows: Take a $K(G_v,1)$ for each orbit of vertices and a $K(G_e,1) \times [0,1]$ for each orbit of edges where $G_v$ and $G_e$ are vertex and edge stabilizers respectively. To construct $X_T$ we take the quotient of these spaces with identifications between the $K(G_v,1)$ and $K(G_e, 1) \times 0$ (or $K(G_e, 1) \times 1$) that induce the monomorphisms: $G_e \mapsto G_v$ \cite{ScottWall1977}. 

 The \textit{free splitting complex}, $FS_n$, of $F_n$ is the simplicial complex where an $n$-simplex is an $(n+1)$-edge free splitting \cite{HandelMosher2013}. The \textit{cyclic splitting graph}, $ FZ_n$,  is the graph where the vertices are $1$-edge cyclic splittings, and two vertices are adjacent when they are both a refinement of the same $2$-edge cyclic splitting \cite{Mann2014}. Note that $FS_n$ is a sub-complex of $FZ_n$ as a cyclic splitting can have trivial edge groups. The \textit{maximally cyclic splitting graph}, $FZ^{max}_n$, is defined as $FZ_n$ along with the additional restriction that the edge stabilizers be closed under taking roots \cite{HorbezWade2015}. Both cyclic graphs function identically in our proof, so we let $\mathcal{Z}_{2g}$ stand in for $FZ_{2g}$ and $FZ^{max}_{2g}$. 

The \textit{arc graph}, $\mathcal{A}(S_g^1)$, of $S_g^1$ is the graph where vertices are free isotopy classes of arcs, and two vertices are adjacent when the two arcs can be realized disjointly on $S_g^1$. The \textit{arc and curve graph}, $\mathcal{AC}(S_g^1)$, of $S_g^1$ is defined in the same way but the vertices range over free isotopy classes of arcs and curves.

\begin{Def} $\psi: \mathcal{A}(S_g^1) \mapsto FS_{2g}$ is given by collapsing a neighborhood of an arc to $\partial S_g^1$ which gives an $X_T$ for a $1$-edge splitting, $T$. \end{Def}

\begin{Def} $\psi_Z: \mathcal{AC}(S_g^1) \mapsto  \mathcal{Z}_{2g}$ is given as $\psi$ for arcs. For curves, take an annulus of the curve to be the edge space of an $X_T$ for a $1$-edge $\mathbb{Z}$-splitting $T$. \end{Def}
We can also view these maps in terms of lifting an arc or curve to the universal cover and letting $T$ be the dual tree of these lifts.
 
In \cite{HamenstädtHensel2015}, Hamenst\"{a}dt and Hensel show that there exists a $1$-Lipschitz retraction of $FS_{2g}$ to $\mathcal{A}(S_g^1)$. We will provide an alternate approach, as well as showing that $\mathcal{AC}(S_g^1)$ is a $(1,C)$-coarse Lipschitz retract of $\mathcal{Z}_{2g}$: 

\begin{customthm}{A} $\mathcal{A}(S_g^1)$ is a $(1,C)$-coarse Lipschitz retract of $FS_{2g}$. In particular, $\psi$ is a $Mod(S_g^1)$-equivariant quasi-isometric embedding of  $\mathcal{A}(S_g^1)$ into $FS_{2g}$.   \end{customthm}

\begin{customthm}{B$'$} $\mathcal{AC}(S_g^1)$ is a $(1,C)$-coarse Lipschitz retract of $FZ_{2g}$. In particular, $\psi_Z$ is a $Mod(S_g^1)$-equivariant quasi-isometric embedding of $\mathcal{AC}(S_g^1)$ into $FZ_{2g}$.   \end{customthm}

\begin{customthm}{B$''$} $\mathcal{AC}(S_g^1)$ is a $(1,C)$-coarse Lipschitz retract of $FZ^{max}_{2g}$. In particular, $\psi_Z$ is a $Mod(S_g^1)$-equivariant quasi-isometric embedding of $\mathcal{AC}(S_g^1)$ into $FZ^{max}_{2g}$.   \end{customthm}

The constant $C$ depends only the genus of the surface; in particular, we have $C=5376(6g-3)^2$ for Theorem A. For Theorem B$'$ \& B$''$, we have $C=21504(6g-3)^2$ .

In Section 2 we prove a pair of key technical lemmas about arcs on $S_g^1$, Section 3 contains the proof of Theorem A, and in Section 4 we provide the necessary changes to Section 3 in order to prove Theorem B$'$ \& B$''$. 
\subsection{Acknowledgments} The author would like to thank Mladen Bestvina for his guidance in this project and Richard D. Wade for his many helpful comments on an earlier draft. The author would also like to thank Morgan Cesa, Radhika Gupta and Kishalaya Saha for their helpful and inspiring conversations. 

\subsection{Retraction} We will now define the retraction, $\phi$; showing $\phi$ is a coarsely well defined map will form the bulk of the paper. Let $T$ be a splitting of $F_{2g}$. As $S_g^1$ and $X_T$ share a fundamental group, we can consider homotopy equivalences between them. If the equivalence, $F$, is transverse to a point, $p$, on the interior of the edge of $X_T$, then $F^{-1}(p)$ will be a $1$-manifold on $S_g^1$. 

\begin{Def} Fix a point, $p$, on the interior of the edge space of $X_T$. Let $\lbrace F_i \rbrace_{i \in I}$ be a collection of homotopy equivalences, $F_i: S_g^1 \mapsto X_T$, such that $F_i$ is transverse to $p$ and $|F_i^{-1}(p) \cap \partial S_g^1|$ is minimal over all homotopy equivalences. \end{Def}
\begin{Def}
$\phi: FS_{2g} \mapsto \mathcal{A}(S_g^1)$ is the collection of arcs contained in a $F_i^{-1}(p)$ over all $i \in I$. \end{Def}

    We note that $F_i^{-1}(p)$ is nonempty since $F_i$, as a homotopy equivalence, must map $S_g^1$ to the edge of any non-trivial $1$-edge splitting. It is possible for $F_i^{-1}(p)$ to contain curves; however, these must be homotopically trivial as they are mapped to a point, and hence can be removed. $F_i^{-1}(p)$ must then be a collection of arcs, so $\phi(T)$ is non-empty for $T \in FS_{2g}$. 

 The word represented by the boundary loop of $S_g^1$ has a unique reduced form with respect to a splitting. The word length of this reduced form is $|F_i^{-1}(p) \cap \partial S_g^1|$ as we took it to be minimal over all equivalences. Since the reduced form is unique, we can then take the $F_i$ to be equal on the boundary. Let $F_{\partial S_g^1}$ be the restriction of the $F_i$ to $\partial S_g^1$. Note the word length of the reduced form, and hence $|F_i^{-1}(p) \cap \partial S_g^1|$, is unbounded over all splittings, and so we can not use this to bound the number of arcs.

   It should be noted that we have no obvious idea of what $\phi(T)$ will look like in $\mathcal{A}(S_g^1)$. During a homotopy between $F_i$ and $F_j$, it is not necessary for the equivalence to remain transverse. Arcs in $F_i^{-1}(p)$ may join and then separate into different arcs during the homotopy, and these new arcs need not be close to arcs in $F_i^{-1}(p)$.

\section{Arcs}

   The goal is of this section is to understand how a power of an arbitrary curve, $\beta$, on $S_{g}^{1}$ can be decomposed into embedded arcs. When $\beta$ is simple, we show that the embedded arcs will be close in $\mathcal{A}(S_g^1)$ (Lemma 10). If $\beta$ is not simple, we will prove that for $n \geq 12$, $\beta^n$ can not be expressed as the union of two embedded arcs (Lemma 9). To show this, we will consider the self intersections of a nice representative of $\beta^n$. The order the self intersections appear in will make it impossible for the curve to be decomposed into two embedded segments. Finally, we will show that the relative order of some these intersections can only be changed via homotopy in ways that are not compatible with a decomposition into arcs. In order to keep track of intersections, we will consider homotopies as a sequence of \textit{Reidemeister moves}:

\begin{figure}[h!]
\centering
\labellist
\pinlabel Monogon: [r] at -20 239
\pinlabel Bigon: [r] at -20 136
\pinlabel Triangle: [r] at -20 34
\endlabellist

\includegraphics[scale=.5]{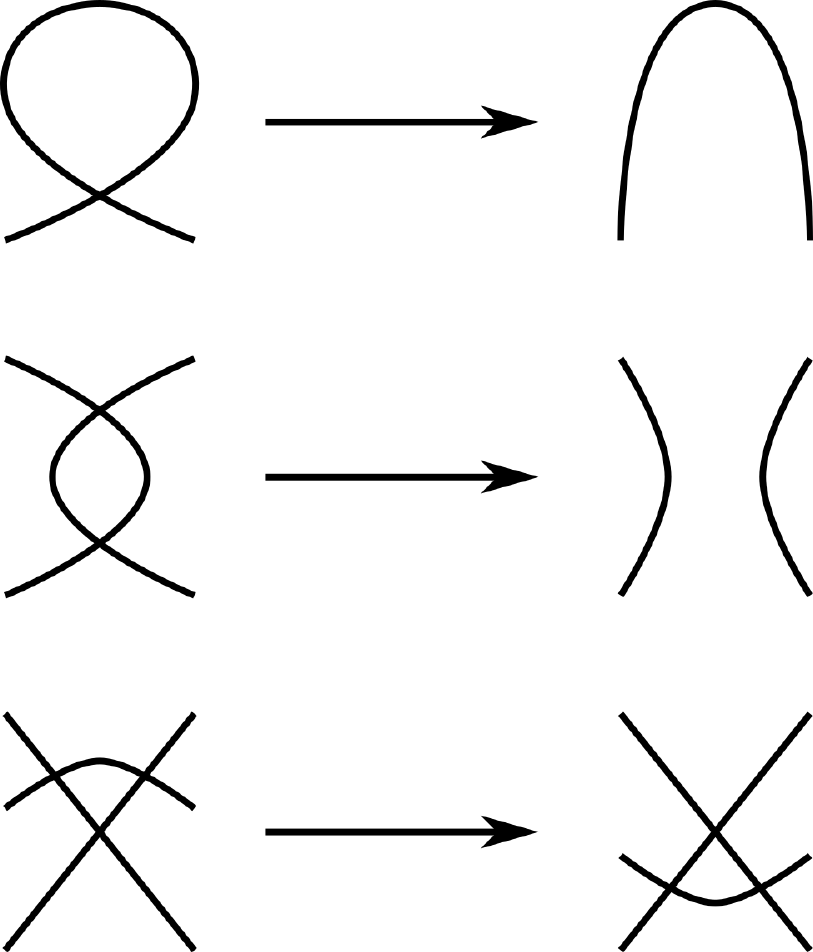}

\end{figure}

The moves occur in a disk which intersects the curve only as pictured. An \textit{embedded triangle} will be as pictured above, i.e bounding a disk, but the curve may intersect an embedded triangle.

\subsection{Coloring of Curves}

The existence of bigons, and hence bigon moves, will be a key technical detail in our proof. This is due to the notion of coloring a curve. A curve (not necessarily simple) on $S_{g}^{1}$ is \textit{two colorable} if it can be decomposed into two segments such that each segment is embedded. Let $c_1, c_2$ be the endpoints of these segments; we can take the $c_i$ such that they are not intersections of the curve. One can homotope any curve by pushing all of its intersections to a short embedded segment which yields an obvious two coloring of the curve. However, this pushing homotopy will form bigons for some curves. If a segment of the curve intersects the bigon region, it is possible that we can not remove the bigon and preserve the two coloring (Figure 1). We note that while bigons affect colorability and are the focus of our next proposition, it is not hard to see that monogons will not affect two colorability.  

\begin{Prop} If $\alpha$ is two colorable and $c_1, c_2 \in \partial S_g^1$, then we can remove all but possibly one bigon, $K_{\alpha}$, such that the the resulting curve, $\alpha'$, is two colorable. \end{Prop}
\begin{proof} 

If we can not remove a bigon, $K$, and preserve colorability, then a $c_i$ either occurs on $\alpha$ in $Int(K)$ or in $\partial K$. If a $c_i$ occurs on $\alpha$ in $Int(K)$, then clearly $c_i \notin \partial S_{g}^{1}$.  

If $c_i \in \partial K$, then $c_1$ and $c_2 \in \partial K$. Moreover, both segments that form $K$ will contain a $c_i$ in its interior; otherwise, one of the segments would be a single color and would intersect the other segment in two colors. There can only be one such $K$, which we call $K_{\alpha}$, as the two segments that form $K_i$ will be determined by $c_1$ and $c_2$. If $K_{\alpha}$ is contained in another bigon, then either $c_1$ or $c_2 \notin \partial S_{g}^{1}$; therefore, we can remove all other bigons without removing $K_{\alpha}$. Removing these bigons preserves colorability, so our new curve, $\alpha'$, is two colorable.    
\end{proof}

\begin{figure}[h!]
\centering

\includegraphics[scale=.6]{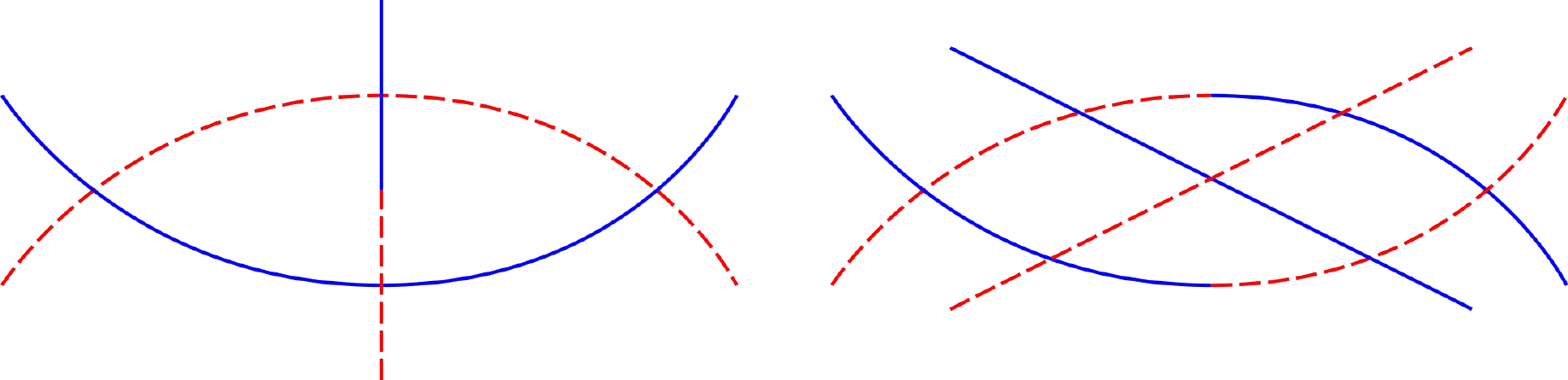}
\caption{The bigon on the right is an example of a $K_{\alpha}$.}
\end{figure}

\subsection{Relative Order}
\begin{Def} Given a curve, $\alpha$, fix a basepoint and orientation. Enumerate the self intersections, $a_1, a_2, \ldots ,a_n$. We record the order of these intersection by assigning a pair, $(k_1, k_2)$, to an intersection $a_k$ where $a_k$ is seen as the $k_1$-th and $k_2$-th intersection.\end{Def}

In practice, we will select a curve, $\alpha$, in minimal position when we enumerate the intersections. Given another representative of the curve, $\alpha'$, we will obtain a map from the $a_k$'s to the $a'_k$'s via the homotopy. We will only be concerned with $a'_k$'s that are associated with $a_k$'s, so for ease of notation we will refer to the $a'_k$'s as $a_k$'s on $\alpha'$. The \textit{relative order} of a collection of intersections refers to the ordering of the $k_1$, $k_2$ associated to the intersections in the cyclic ordering, $\prec_c$, on $1,2...2n$. We will be concerned about whether we see an intersection, $a_k$, twice, before seeing another intersection, $a_m$, once. We say a collection of intersections is \textit{consecutively ordered} if for any $a_k$, $a_m$ in the collection with pairs $(k_1, k_2), (m_1, m_2)$, we have $k_1 \prec_c k_2 \prec_c m_1 \prec_c m_2$ up to permutation of $k_1$ with $k_2$ and $m_1$ with $m_2$. Note that a collection being consecutively ordered only depends on the relative order of the collection, not of the whole curve.

  \begin{Prop} Let $\alpha$ be a curve with self intersections $a_1,a_2, \ldots ,a_n$. If $\alpha'$ is a curve obtained from $\alpha$ by a sequence of triangle moves, then the relative order of $a_i$ and $a_j$ can be altered only if there exists an embedded triangle with $a_i$ and $a_j$ as vertices in $\alpha$.\end{Prop}
\begin{proof}

 Since a triangle move occurs in a disk that only intersects the curve in that triangle, a triangle move will only change the relative order of the vertices of the triangle. Therefore, if $a_i$ and $a_j$ have their relative order changed, then at some point in the homotopy $a_i$ and $a_j$ must be vertices of the same embedded triangle. This embedded triangle will also exist in $\alpha$ since triangle moves do not remove intersections and disks will remain disks throughout a homotopy.
\end{proof}

\subsection{Perturbed Geodesics}
We will use notation and a result from \cite{deGraafSchrijver1997} for our next proposition. The main result of \cite{deGraafSchrijver1997} is to put a collection of curves on a triangulizable surface into minimal position without creating bigons or monogons at any point in the homotopy. In order to accomplish this for hyperbolic surfaces, they prove that one can take curves in the class of $\beta^n$ and homotope them into a perturbed geodesic form without the creation of bigons or monogons. We will now introduce the notation to make the notion of a perturbed geodesic precise.

 We consider a geodesic curve, $\beta$, as a graph on the surface with intersection points as vertices. Now we take a polygonal decomposition of a neighborhood of $\beta$ in the following fashion: For each vertex we choose a convex polygon, $P_v$, containing $v$ in its interior, and for each edge $e$ a convex $4$-gon, $P_e$, such that any edge $e=uv$ is contained in $P_u \cup P_e \cup P_v$. We assume that the $P_v$ are mutually disjoint and that $P_e$ are mutually disjoint, while $P_v$ and $P_e$ intersect if and only if $v$ is incident with $e$. In which case, $P_v$ and $P_e$ intersect in a side both of $P_e$ and of $P_v$. Moreover, each side of any $P_v$ is equal to the intersection of $P_v$ with $P_e$ for some edge $e$ incident with $v$. We can also assume that if $v$ and $v'$ are the vertices incident with the edge $e$, then $P_v$ and $P_{v'}$ intersect $P_e$ in opposite sides of $P_e$. Let $\beta$ form the circuit $(v_0,e_1,v_1, \ldots ,e_t, v_t)$ in the graph, with $v_0=v_t$.

\begin{Prop}Let $\alpha$ be a representative of $\beta^n$ in minimal position. We can homotope $\alpha$ with a sequence of triangle moves such that there exists intersections  $b_1, b_2, \ldots ,b_{\lfloor \frac{n}{2} \rfloor}$ that are not pairwise connected along the curve by an embedded segment. Furthermore, we may assume that the $b_i$ are consecutively ordered.\end{Prop}

\begin{proof} 

From the proof of the Proposition 14 in \cite{deGraafSchrijver1997}, we see that we can homotope $\alpha$ without creating bigons, such that is it contained in the polygonal neighborhood of $\beta$ where every intersection of $\alpha$ is contained in a $P_v$. Moreover, $\alpha$ traverses, in order, $P_{v_0}, P_{e_1}, P_{v_1}, \ldots, P_{e_t}, P_{v_t}$, $n$ times. Each $P_{v_i}$ occurs twice in the circuit since it is an intersection of $\beta$. We let the $b_i$ be intersections formed by $\alpha$ entering $P_{v_0}$ in the same circuit. In particular, we let $b_i$  be obtained from the $2i\textsuperscript{th}$ circuit which yields $\lfloor \frac{n}{2} \rfloor$ possible $b_i$. We note that any segment connecting $b_i$, $b_j$ will traverse $P_{v_0}, P_{e_1}, P_{v_1}, \ldots, P_{e_t}, P_{v_t}$ and therefore will not be embedded. Also since the $b_i$ occur in the same circuit, they will be consecutively ordered.
\end{proof}

\subsection{Proof of Lemmas 9 \& 10}
\begin{Lemma} If $\beta$ is a primitive non-embedded curve, then $\beta^n$ is not the union of two embedded arcs for $n\geq12$. \end{Lemma}

\begin{proof}

 If a curve can be written as the union of two embedded arcs, then that decomposition will result in a two coloring of the curve with $c_1, c_2 \in \partial S_{g}^{1}$. Let $\gamma$ be a representative of $\beta^n$ that is two colored in such a way. We can remove all bigons, other than possibly $K_{\gamma}$, to get the two colored curve $\gamma'$ (Proposition 5).

 We now homotope $\gamma'$ into a neighborhood of the geodesic representative of $\beta$ in the following manner. Our first step will be to remove $K_{\gamma'}$, if it exists; this will induce triangle moves. We then use Proposition 8 to finish homotopy using only triangle moves. Let $\gamma''$ be the homotoped $\gamma'$. If $n\geq12$, then there exist consecutively ordered intersections, $b_1, b_2, \ldots ,b_6$, on $\gamma''$. 

Now we will show that $b_1,b_3,b_5$ are consecutively ordered in $\gamma'$ by following the homotopy from $\gamma''$. The $b_i$ are not pairwise vertices of an embedded triangle in $\gamma''$ because any edge connecting two of them will not be embedded. By Proposition 7, the $b_i$ are still consecutively ordered after applying the triangle moves coming from Proposition 8. If necessary, we will finish the homotopy with the bigon move associated to $K_{\gamma'}$ and then its induced triangle moves. The segments between $b_1,b_3,b_5$ will each contain both parts of the intersection of either $b_2$, $b_4$, or $b_6$ since the $b_i$ are consecutively ordered. Therefore, $b_1, b_3, b_5$ are not pairwise vertices of an embedded triangle, and hence they are consecutively ordered after the induced triangle moves (Proposition 7). 

For any intersection $a_k=(k_1,k_2)$, we note that there must exist a $c_i$ on the curve between $k_1$ and $k_2$; otherwise, $a_k$ will be an intersection in the same color. Therefore, $b_1, b_3, b_5$ being consecutively ordered implies there exists $c_1, c_2, c_3$ which is a contradiction to $\gamma'$ being two colorable.
\end{proof}

\begin{Lemma} Let $\beta$ be an embedded curve on $S_g^1$ which is not homotopic to $\partial S_g^1$. If $\beta^n$ is the union of two embedded arcs, $\alpha$ and $\gamma$, then $d(\alpha, \gamma)\leq2$ in $\mathcal{A}(S_g^1)$. \end{Lemma}

\begin{proof} 

Since $\alpha \cup \gamma$ is homotopic to $\beta^n$ as a curve, we can represent $\alpha$ and $\gamma$ as $\beta^n$ along with two paths, $h_1$ and $h_2$, from $\partial S_g^1$ to $\beta^n$. In particular, $\alpha$ is $h_1$, then a part of $\beta^n$, then $\bar{h}_2$, and $\gamma$ is $h_1$, then the rest of $\beta^n$, then $\bar{h}_2$. Let $\beta^n$ be contained in an annular neighborhood, $N_{\beta}$, of $\beta$ such that $h_1$ and $h_2$ terminate on $\partial N_{\beta}$. We claim $h_i$ does not intersect $N_{\beta}$, or more precisely that $h_i$ is a single arc on $S_g^1 \setminus N_{\beta}$. If this were not the case, then $h_i$ restricted to $N_{\beta}$ would witness an intersection $h_i$ with either $\alpha$ or $\gamma$. Whether $h_i$ intersects $\alpha$ or $\gamma$ depends on the homotopy class of $h_i$, and so we have a non-removable self intersection of $\alpha$ or $\gamma$. By the same token, $h_1$ and $h_2$ do not intersect themselves or each other.    

Consider the neighborhood of $\alpha \cup \gamma$ comprised from tubular neighborhoods of $h_1$, $h_2$ together with $N_{\beta}$. The boundary of this neighborhood will be a collection of arcs and possibly the curve $\beta$. Any curves other than $\beta$ would imply an intersection of $h_1$ or $h_2$. The arcs are clearly disjoint from $\alpha$ and $\gamma$, and since $\beta$ is not homotopic to $\partial S_g^1$, we have at least one of these arcs is non-trivial.
\end{proof}

\section{Free Splittings}
   The main focus of this section will be proving that $\phi$ is coarsely well defined (Lemma 19). Recall that $\phi(T)$ for $T \in FS_{2g}$ is the set of arcs on $S_g^1$ that arise as the preimage of a particular marked point, $p$, on a total space, $X_T$. These preimages will range over different homotopy equivalences, $\lbrace F_i \rbrace_{i \in I}$, between $S_g^1$ and $X_T$, and arcs resulting from different $F_i$ may intersect. We wish to show that the set of arcs for any splitting has bounded diameter in $\mathcal{A}(S_g^1)$ depending only on the genus.

   An \textit{arc system} denotes the collection of arcs contained in $F_i^{-1}(p)$ for a single equivalence. When we discuss arcs, we will consider them as fixed endpoint homotopy classes. Note that it is possible to have two freely homotopic arcs as distinct members of an arc system; a \textit{parallel family} is a maximal collection of these freely homotopic arcs in an arc system.

   Rather than considering all of the arcs comprising an arc system, it will be sufficient to gain an understanding of a large parallel family in the arc system. Our main tool to restrict these large families will use the word represented by the boundary loop, and the following property of parallel families: A parallel family of arcs will form disks with segments of the boundary; the arcs are mapped to $p$ by the $F_i$, so the boundary segments at each side of a parallel family must be mapped to inverse words. Since the $F_i$ agree on $\partial S_g^1$, we can then use inverse pairs that appear in the word represented by the boundary loop to limit parallel families over all $F_i$. In particular, we will think of such inverse pairs as pieces of the boundary that can potentially be connected by a parallel family. Large parallel families will require large inverse pairs, and, together with the results of Section 2, will give us a bound on the number of $2$-balls needed to contain $\phi(T)$. Finally, we will show any two points in $\phi(T)$ are connected by a path in $\phi(T)$ which we construct via the homotopy between $F_i$'s.    

\subsection{Preliminaries}

We now set up notation for viewing the boundary as composed of a finite number of pieces that can possibly be connected with parallel families of arcs. Recall that $F_{\partial S_g^1}$ is the restriction of the $F_i$ to the $\partial S_g^1$, and hence we can think of $F_{\partial S_g^1}^{-1}(p)$ as the possible endpoints of the arcs.

\begin{Def} An endpoint set, $\Lambda$, is a subset of $F_{\partial S_g^1}^{-1}(p)$ that is connected in the sense that there exists a segment on $\partial S_g^1$ that contains $\Lambda$, but no other points in $F_{\partial S_g^1}^{-1}(p)$. The minimal segment that contains $\Lambda$ is mapped to a loop based at $p$ by the $F_i$. The subword associated to $\Lambda$, $w_{\Lambda}$, is the word represented by this loop with orientation inherited from $\partial S_g^1$. \end{Def}

We think of the points of $\Lambda$ as half the endpoints of a parallel family of arcs.

\begin{Def} A partial arc system, $(\Lambda,\Upsilon)$, is a pair of two disjoint endpoint sets, $\Lambda$ and $\Upsilon$, such that $w_{\Lambda} = w_{\Upsilon}^{-1}$. \end{Def}

Note $\Lambda$ and $\Upsilon$ have the same cardinality; their minimal containing segments are mapped to inverse words and hence cross $p$ the same number of times. 

\begin{Prop} The endpoints of every parallel family in an an arc system form a partial arc system. \end{Prop}
\begin{proof} If the parallel family only contains one arc, it is clear that the two endpoints form a partial arc system. If there are at least two arcs in the parallel family, then these arcs, along with segments of $\partial S_g^1$, will form disks. We claim there is a maximal disk for a parallel family. There clearly exists a maximal disk for two arcs, and if we induct on the number of arcs, we see that the new arc is either contained in or disjoint from the disk. In the latter case, we can extend the disk via a homotopy between the new arc and an arc contained in the disk. The segments of $\partial S_g^1$ that in part form the maximal disk will be our minimal containing segments. Since the two arcs that in part form the maximal disk are mapped by $F_i$ to $p$, one of the minimal containing segments must be mapped to the inverse of the other in order for the circle bounding the disk to have a null-homotopic image.

If there was a point of $F_{\partial S_g^1}^{-1}(p)$ on one of the segments that was not an endpoint of our parallel family, it would have to be an endpoint of some other arc. In order for that arc to be non-trivial and not homotopic to arcs in the parallel family, it would have to leave the disk which would result in two arcs of the arc system intersecting. Therefore, the endpoints of the parallel family form a partial arc system. 
\end{proof}

\begin{Def} Let $(\Lambda,\Upsilon)$ be a partial arc system which arises from a parallel family which contains an arc, $\alpha$. The arc associated to $(\Lambda,\Upsilon)$ is the free homotopy class of $\alpha$. \end{Def}

   We note that not every partial arc system will arise from a parallel family, and so not every partial arc system will have an associated arc.

\begin{Prop} The associated arc of a partial arc system, $(\Lambda, \Upsilon)$, is unique if it exists. \end{Prop}
\begin{proof}Since $S_g^1$ is orientable, there is exactly one way to pair the points of $\Lambda$ and $\Upsilon$ with non intersecting, freely homotopic arcs. Let $\alpha \in  F_i^{-1}(p)$ and $\alpha' \in F_j^{-1}(p)$ be arcs that share endpoints in $F_{\partial S_g^1}^{-1}(p)$. We claim $\alpha$ and $\alpha'$ are homotopic: Consider the two loops formed by taking a segment of $\partial S_g^1$ along with either $\alpha$ or $\alpha'$. These loops will have homotopic images since $F_i(\alpha)=F_j(\alpha')=p$ and $F_i$ and $F_j$ agree on $\partial S_g^1$. Therefore, there is exactly one way to pair the points of $\Lambda$ and $\Upsilon$ and exactly one homotopy class of arcs for a pair of endpoints. 
\end{proof}

\subsection{Proof of Main Result}
   If we fix one of the endpoint sets of a partial arc system and vary the second endpoint set, we will obtain distinct partial arc systems. For these endpoint set pairs to be partial arc systems, they must of course have inverse subwords. In the case that these varied endpoint sets intersect, we can then say that the subword will be roughly periodic:

\begin{Prop} Let $(\Lambda, \Upsilon)$ and $(\Lambda, \Theta)$ be partial arc systems. If $|\Upsilon \setminus \Theta| \leq \frac{1}{n}(|\Lambda|-1)$, then $w_{\Upsilon}=w_{\Theta}= w^nv$ where $v$ is a prefix of $w \in F_{2g}$. 
 \end{Prop}

\begin{proof} 

Let $w_{\Lambda}^{-1}=w_{\Upsilon}=w_{\Theta}= x_1x_2 \ldots x_l$ where the $x_i$ are words associated to sub-endpoint sets which contain exactly two points. As $\Upsilon$ and $\Theta$ overlap as endpoint sets, it follows that $x_i=x_{i+k}$ for $k= |\Upsilon \setminus \Theta|$ and $i \leq l-k$. Therefore, $w_{\Upsilon}=w_{\Lambda}= w^nv$ where $w=x_1x_2 \ldots x_ k$ and $v=x_1x_2 \ldots x_m$ for $m <k$.
\end{proof}

\begin{Prop} Let $(\Lambda, \Upsilon_0)$ be a partial arc system and consider all partial arc systems, $(\Lambda,\Upsilon_i)$, such that $|\Upsilon_0 \cap \Upsilon_i| \geq \frac{1}{2}|\Upsilon_0|$. There are either at most $12$ arcs  associated to the $(\Lambda, \Upsilon_i)$, or the associated arcs can be contained in a $2$-ball in $\mathcal{A}(S_g^1)$.  \end{Prop}

\begin{proof}

Let $\Upsilon_j$ have largest, non equal, intersection with $\Upsilon_0$ over the $\Upsilon_i$. By applying Proposition 16 to $\Upsilon_0$ and $\Upsilon_j$, we have $w_{\Upsilon_j} = w_{\Upsilon_i} = w^nv=(x_1x_2 \ldots x_k)^nv$. If $w= x^z$ for some $x \in F_{2g}$, then $x=x_1x_2\ldots x_{\frac{k}{z}}$ with $\frac{k}{z} \in \mathbb{Z}$. Otherwise for some $i\leq k$, $x_ix_{i+1}$ can be represented in $\pi_1(X_T)$ by a loop that does not cross $p$ which violates the the minimality of $|F_{\partial S_g^1}^{-1}(p)|$. As $w=(x_1x_2\ldots x_{\frac{k}{z}})^z$, we can find an $\Upsilon_i$ that has larger intersection with $\Upsilon_0$ than does $\Upsilon_j$ when $z>1$. Thus, $w$ is not periodic. 

Given an $\Upsilon_i$, we claim that it is shifted from $\Upsilon_0$ by $w^m$, i.e. the first $|\Upsilon_0 \setminus \Upsilon_i|+1$ points of $\Upsilon_0$ is an endpoint set with associated subword $w^m$. If not, then $w_{\Upsilon_i} = (x_hx_{h+1} \ldots x_{h+k-1})^nv = (x_1x_2 \ldots x_k)^nv$ for $1<h<k$, and so $w=(x_h \ldots x_k)(x_1 \ldots x_{h-1})=(x_1 \ldots x_{h-1})(x_h \ldots x_k)$ which implies $w$ is periodic.  

Now consider the loop formed by the union of the minimal containing segment of the first $|\Upsilon_0 \setminus \Upsilon_i|+1$ points of $\Upsilon_0$, and the arcs associated to the $(\Lambda, \Upsilon_0)$ and $(\Lambda, \Upsilon_i)$ partial arc systems. The minimal containing segment will be mapped to $w^m$, and the arcs will be mapped to paths homotopic to the constant path $p$. So our loop on $S_g^1$ will have the homotopy type of $w^m$. Now we apply the results of Section 2 to say that either the two arcs are contained in the same $1$-ball (Lemma 10), or one of the arcs is not embedded (Lemma 9), or $m<12$ (Lemma 9).Therefore, there are either at most $12$ associated arcs or they can all be contained in a $2$-ball.
\end{proof}

We can think of an arc system as being comprised of partial arc systems formed by its parallel families. We will abuse notation slightly and say that an endpoint set, $\Upsilon$, is contained in an arc system if $\Upsilon \subseteq \Lambda$ or $\Theta$ for one of these partial arc systems, $(\Lambda, \Theta)$. 

\begin{Prop} There exists a collection $\Lambda_1,\Lambda_2,...\Lambda_{4(6g-3)}$ of endpoint sets with $|\Lambda_i| \geq \frac{|F_{\partial S_g^1}^{-1}(p)|}{4(6g-3)}$ such that every arc system contains a $\Lambda_i$. \end{Prop}

\begin{proof}
 Let the $\Lambda_i$ be a partition of $F_{\partial S_g^1}^{-1}(p)$ with $|\Lambda_i|=\left \lceil \frac{|F_{\partial S_g^1}^{-1}(p)|}{4(6g-3)} \right\rceil$. By a standard Euler characteristic argument, there exist at most $6g-3$ parallel families in any arc system, which partition $F_{\partial S_g^1}^{-1}(p)$. Hence, one such family contains at least $\frac{|F_{\partial S_g^1}^{-1}(p)|}{6g-3}$ points of $F_{\partial S_g^1}^{-1}(p)$, and this family forms a partial arc system, $(\Lambda, \Theta)$, with $|\Lambda|=|\Theta|\geq  \frac{|F_{\partial S_g^1}^{-1}(p)|}{2(6g-3)}$ (Proposition 13). Thus, both $\Lambda$ and $\Theta$ contain a $\Lambda_i$. 
\end{proof}

\begin{Lemma}  $\phi: FS_{2g} \mapsto \mathcal{A}(S_g^1)$ is coarsely well defined. \end{Lemma}

\begin{proof} 

Every arc system contains a $\Lambda_i$, and hence every arc system will contain an arc associated to a partial arc system of the form $(\Lambda_i, \Upsilon_{i,j})$ for some collection of endpoint sets $\Upsilon_{i,j}$. 
We note that for all $j$, $|\Upsilon_{i,j}|=|\Lambda_i| \geq \frac{|F_{\partial S_g^1}^{-1}(p)|}{4(6g-3)}$ (Proposition 18). Hence, there is a collection of at most $8(6g-3)$ $\Upsilon_{i,j}$'s such that each pairwise intersection has cardinality less than $\frac{1}{2}|\Lambda_i|$. By applying Proposition 17 to any other $\Upsilon_{i,j}$, we see that each $\Lambda_i$ will contribute arcs that can be contained in $96(6g-3)$ $2$-balls. The arcs that form an arc system do not intersect, so the image of $\phi$ can be contained in $1$-balls around these associated arcs taken over all the $\Lambda_i$. Therefore, we can contain $\phi(T)$ in $384 (6g-3)^2$ $3$-balls in $\mathcal{A}(S_g^1)$.

   To show $\phi(T)$ has bounded diameter, we will show $\phi(T)$ is connected. Given two arcs, $\alpha \in F_i^{-1}(p)$ and $\beta \in F_j^{-1}(p)$, we will form a path by following the homotopy between $F_i$ and $F_j$. During the homotopy, it is possible that the equivalence map is no longer transverse, and hence we see a picture where two arcs meet and separate into different arcs. In terms of $ \mathcal{A}(S_g^1)$, we can view the homotopy as a sequence of arc switching. We can arrange for any collection of arcs that meet during the homotopy to intersect in exactly one point, and hence at least one original arc will not intersect the new arcs. By altering the local speed of the homotopy if necessary, we can ensure these new arcs are in $\phi(T)$. The maximal length of such a path will be $2688 (6g-3)^2$, which bounds the diameter of $\phi(T)$.
\end{proof}

\begin{customthm}{A} $\mathcal{A}(S_g^1)$ is a $(1,C)$-coarse Lipschitz retract of $FS_{2g}$. In particular, $\psi$ is a $Mod(S_g^1)$-equivariant quasi-isometric embedding of  $\mathcal{A}(S_g^1)$ into $FS_{2g}$.   \end{customthm}

\begin{proof}

Recall that for an arc $\alpha$, $\psi(\alpha)$ is the splitting associated to the total space formed by collapsing a neighborhood of $\alpha$ to $\partial S_g^1$. Such a collapse is a homotopy equivalence, and the preimage of a point on the edge will naturally return $\alpha$, and only $\alpha$. If $\psi(\alpha)=T$, then we have $|F_{\partial S_g^1}^{-1}(p)|=2$, and so $\phi(T)=\alpha$.

 If we have two adjacent $1$-edge splittings, $T_1$ and $T_2$, then they are a common refinement of a $2$-edge splitting, $T_3$. Consider $X_{T_3}$ with two marked points, one on each edge and a transverse homotopy equivalence between $S_g^1$ and $X_{T_3}$. If we factor the equivalence through the collapse maps of $X_{T_3}$ to $X_{T_1}$ and $X_{T_2}$, then we see that the preimage of each point is a collection of arcs contained $\phi (T_1)$ and $\phi (T_2)$ respectively. These arcs will not intersect as they are contained in preimages of distinct points, and so $diam_{\mathcal{A}(S_g^1)}(\phi(T_1) \cup \phi(T_2)) \leq d_{FS_{2g}}(T_1,T_2) + 5376(6g-3)^2$ (Lemma 19). Therefore, $\phi$ is a $(1,C)$-coarse Lipschitz retraction of $FS_{2g}$ onto $\psi(\mathcal{A}(S_g^1))$ for $C=5376(6g-3)^2$. In particular, the existence of a coarse retraction onto a subspace easily implies that the inclusion map, in this case $\psi$, is a quasi-isometric embedding.
\end{proof}

\section{Cyclic Splittings}
   Showing that $\mathcal{AC}(S_g^1)$ is a coarse Lipschitz retract of $\mathcal{Z}_{2g}$ will proceed in a very similar manner as Theorem A. With this in mind, we will address the differences in this section rather than repeating all of Section 3. The main hurdle will be to expand the definition of $\phi$ to non-trivial cyclic splittings in a way that is compatible with our previous arguments.

\subsection{Preliminaries}

A $1$-edge cyclic splitting has one of the following forms \cite{HorbezWade2015}:

\begin{itemize}
\item A separating $1$-edge free splitting, $F_n = A * B$, where $A$ and $B$ are complementary proper free factors of $F_n$.

\item A non-separating $1$-edge free splitting, $F_n= C*$, where $C$ is a rank $n-1$ free factor of $F_n$. 

\item A separating $1$-edge $\mathbb{Z}$-splitting, $F_n= A*_{\langle y \rangle}  (B*\langle y \rangle)$, where $A$ and $B$ are complementary proper free factors of $F_n$ and $y\in A$.

\item A non-separating $1$-edge $\mathbb{Z}$-splitting, $F_n= (C*\langle y^t \rangle)*_{\langle y \rangle}$, where $C$ is a rank $n-1$ free factor of $F_n$, $y \in C$ and $t$ denotes the stable letter. 

\end{itemize}

   As before, in order to define our retraction we will consider homotopy equivalences between $S_g^1$ and $X_T$. If $F$ is such an equivalence and $T$ is a $1$-edge $\mathbb{Z}$-splitting, we will consider $F^{-1}(c)$ where $c$ is an embedded curve on the interior of the edge space of $X_T$ (the edge space is homeomorphic to  $S^1 \times [0,1]$).

\begin{Def} For a $1$-edge $\mathbb{Z}$-splitting, $T$, fix a curve, $c$, on the interior of the edge space of $X_T$. Let $\lbrace F_i \rbrace_{i \in I}$ be a collection of homotopy equivalences, $F_i: S_g^1 \mapsto X_T$, such that $F_i$ is transverse to $c$, $|F_i^{-1}(p) \cap \partial S_g^1|$ is minimal over all equivalences, and $F_i(\partial S_g^1) \cap c$ is a single point, $p$.

We also have the following additional restriction on the $F_i$ for a separating splitting, $A*_{\langle y \rangle}  (B*\langle y \rangle)$: words in $(B*\langle y \rangle)$ which are represented by segments of the boundary mapped to loops based at $p$ begin and end in $B$. \end{Def}

\begin{Def} Let $\phi_Z: \mathcal{Z}_{2g} \mapsto \mathcal{AC}(S_g^1)$ be $\phi$ for $1$-edge free splittings. Given a $1$-edge $\mathbb{Z}$-splitting, $\phi_Z$ is the collection of arcs and curves contained in a $F_i^{-1}(c)$ over all $i\in I$.\end{Def}

  If $T$ is a separating splitting, the segments of the boundary mapped to loops based at $p$ will alternate between words in $A$ and words in $ (B*\langle y \rangle)$. For a word in $ (B*\langle y \rangle)$, we can push its initial and terminal $y$'s into $A$. With these restrictions on the $F_i$, we can take the $F_i$ to agree on $\partial S_g^1$. 
\subsection{Alterations}
Now that arcs can now be mapped to $y^j$ for $j \in \mathbb{Z}$, we need to rework Proposition 13 for $\mathbb{Z}$-splittings. 

\begin{Prop} The endpoints of every parallel family in an an arc system, possibly excluding the endpoints of two of the arcs, form a partial arc system. Furthermore, the images of these arcs are trivial loops.   \end{Prop}

\begin{proof}

Let $\Lambda_1$ and $\Lambda_2$ be a pair of endpoint sets that arise as the endpoints of exactly two arcs of the parallel family. These two arcs together with the minimal containing segments of $\Lambda_1$ and $\Lambda_2$ form a disk. The disk gives us the relation $w_{\Lambda_1}^{-1}=y^k w_{\Lambda_2} y^j$, where the arcs are mapped to $y^k$ and $y^{j}$ for $k,j \in \mathbb{Z}$ up to orientation of the arcs. Let $\alpha$ be an arc that in part forms two such endpoint set pairs: $\Lambda_1$, $\Lambda_2$ and $\Upsilon_1$, $\Upsilon_2$ (Figure 2).

\begin{figure}[h!]

\centering

\labellist \small 
\pinlabel $\Lambda_1$ at 68 154
\pinlabel $\Upsilon_1$ at 204 154
\pinlabel $\Lambda_2$ at 68 -15
\pinlabel $\Upsilon_2$ at 204 -15
\pinlabel $\alpha$ [l] at 140 70

\endlabellist

\includegraphics[scale=.5]{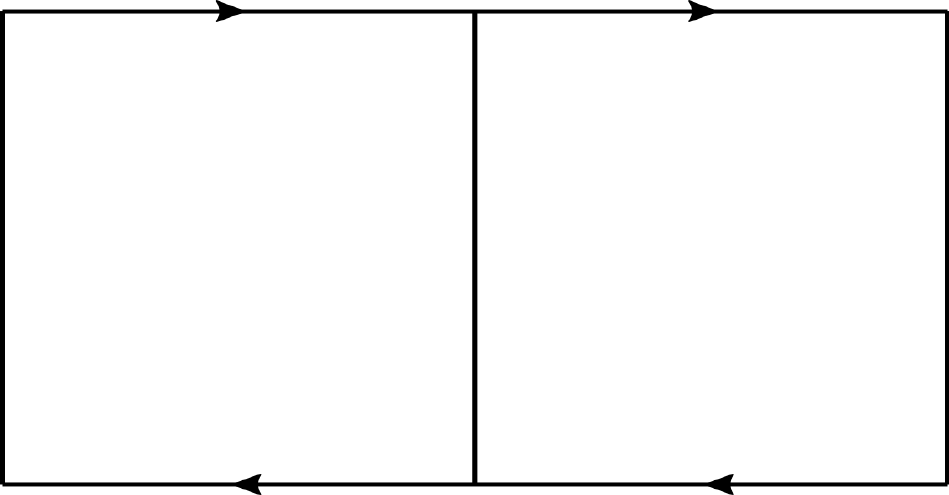}
 \setlength{\abovecaptionskip}{20pt} 
\caption{}

\end{figure}

   For a separating $\mathbb{Z}$-splitting, $A*_{\langle w \rangle}  (B*\langle y \rangle)$, we can take $w_{\Lambda_1} \in B*\langle y \rangle$  without loss of generality as $w_{\Lambda_1}$, $w_{\Upsilon_1}$ alternate in $A$ and $B*\langle y \rangle$. Since $w_{\Lambda_1}$ begins and ends in $B$, $w_{\Lambda_1}^{-1}=y^k w_{\Lambda_2} y^j$ implies $k=j=0$, and so $\alpha$ is mapped to a trivial loop. 

   Without loss of generality, if we have a non-separating $\mathbb{Z}$-splitting, $(C * \langle y^t  \rangle)*_{\langle y \rangle}$, then either $w_{\Lambda_1}$ ends in $t$ or $w_{\Upsilon_1}$ begins in $t$. Otherwise $|F_i^{-1}(p) \cap \partial S_g^1|$ would not be minimal as we could remove the shared point of $\Lambda_1$ and $\Upsilon_1$. If $w_{\Lambda_1}$ ends in $t$ then $w_{\Lambda_1}^{-1}=y^k w_{\Lambda_2} y^j$ implies $k=0$ and if $w_{\Upsilon_1}$ begins in $t$ then $w_{\Upsilon_1}^{-1}=y^m w_{\Upsilon_2} y^{-k}$ implies $k=0$. Note that $\alpha$ corresponds to $w^k=e$.

Now all the arcs, except for the two arcs which form the maximal disk, are mapped to trivial loops. This allows us to proceed with the proof of Proposition 13 with these two arcs excluded.
\end{proof}
In light of the differences between Proposition 22 and Proposition 13, we will now consider an altered form of Proposition 18:

\begin{Prop} There exists a collection $\Lambda_1,\Lambda_2,...\Lambda_{8(6g-3)}$ of non-empty endpoint sets with $|\Lambda_i| \geq \frac{|F_{\partial S_g^1}^{-1}(p)|}{4(6g-3)}-1$ such that every arc system contains a $\Lambda_i$. \end{Prop}

\begin{proof}

 Let the $\Lambda_i$ be a partition of $F_{\partial S_g^1}^{-1}(p)$ with $|\Lambda_i|=\left \lceil \frac{|F_{\partial S_g^1}^{-1}(p)|}{4(6g-3)} \right\rceil-1$, when $\left \lceil \frac{|F_{\partial S_g^1}^{-1}(p)|}{4(6g-3)} \right\rceil <2$, we instead let $|\Lambda_i|=1$. We need at most $8(6g-3)$ $\Lambda_i$'s. As before, there exist at most $6g-3$ parallel families in any arc system, which partition $F_{\partial S_g^1}^{-1}(p)$. Hence, one such family contains at least $\frac{|F_{\partial S_g^1}^{-1}(p)|}{6g-3}$ points of $F_{\partial S_g^1}^{-1}(p)$, and this family forms a partial arc system, $(\Lambda, \Theta)$, with $|\Lambda|=|\Theta|\geq  \frac{|F_{\partial S_g^1}^{-1}(p)|}{2(6g-3)}-2$ (Proposition 22). Thus, both $\Lambda$ and $\Theta$ contain a $\Lambda_i$. 
\end{proof}

\begin{customthm}{B} $\mathcal{AC}(S_g^1)$ is a $(1,C)$-coarse Lipschitz retract of $\mathcal{Z}_{2g}$. In particular, $\psi_Z$ is a $Mod(S_g^1)$-equivariant quasi-isometric embedding of $\mathcal{AC}(S_g^1)$ into $\mathcal{Z}_{2g}$.   \end{customthm}

\begin{proof}
In order to prove the theorem, we can follow the proof in Section 3, replacing Propositions 13 \& 18 with Propositions 22 \& 23 respectively. Lemma 19 proceeds as before with the possible addition of curves in the image of $\phi_Z$: There can be only one curve, $\gamma$, in $F^{-1}_j(c)$ for a $1$-edge $\mathbb{Z}$-splitting, $T$, because $\gamma$ corresponds to the cyclic letter, $y$, of $T$. If no arcs are contained in $F^{-1}_j(c)$, then there will be no arcs contained in any of the $F^{-1}_i(c)$, so $\phi_Z(T)=\gamma$. Otherwise, $d(\gamma, \alpha) =1$ for an arc, $\alpha \in F^{-1}_j(p)$. Therefore, the addition of curves in the image of $\phi_Z$ does not necessitate any further changes. Finishing the proof in Section 3 yields that $\phi_Z$ is a $(1,C)$-coarse Lipschitz retraction of $\mathcal{Z}_{2g}$ onto $\psi_Z(\mathcal{AC}(S_g^1))$. Using Propositions 22 \& 23 alters the bounds, and we have $C=21504(6g-3)^2$.
\end{proof}

\bibliography{general}
\bibliographystyle{amsalpha}

\end{document}